\documentclass[reqno]{amsart}
\usepackage{amsmath,amsfonts,amsthm,amsopn,color,amssymb,enumitem}
\usepackage{palatino}
\usepackage{graphicx}
\usepackage[english]{babel}
\usepackage{caption}
\usepackage{subcaption}
\usepackage[colorlinks=true]{hyperref}
\hypersetup{urlcolor=blue, citecolor=red, linkcolor=blue}
\usepackage[utf8]{inputenc}

\usepackage{esint}
\usepackage[title]{appendix}

\newtheorem{theorem}{Theorem}[section]
\newtheorem{lemma}[theorem]{Lemma}

\newtheorem{proposition}[theorem]{Proposition}

\def\beq{\begin{equation}}
\def\eeq{\end{equation}}
\def\ba{\begin{array}}
\def\ea{\end{array}}

\def\R{\mathbb R}

\def\e{\varepsilon}

\newcommand{\rmnote}[1]{}

\numberwithin{equation}{section}
\newenvironment{abs}{\textbf{Abstract.}\mbox{  }}{ }
\newenvironment{key words}{\textbf{Keywords}\mbox{  }}{ }

\begin{document}



\title{Modica type estimates and curvature results for overdetermined elliptic problems}

\author{David Ruiz}
\address{(D.~Ruiz)
	IMAG, Departamento de An\'alisis matem\'atico, Universidad de Granada,
	Campus Fuente-nueva, 18071 Granada, Spain} \email{daruiz@ugr.es}

\author{Pieralberto Sicbaldi}
\address{(P.~Sicbaldi)
	IMAG, Departamento de An\'alisis matem\'atico,
	Universidad de Granada,
	Campus Fuentenueva,
	18071 Granada,
	Spain \& Aix Marseille Universit\'e - CNRS, Centrale Marseille - I2M, Marseille, France}
\email{pieralberto@ugr.es}

\author{Jing Wu}
\address{(J.~Wu)
	Departamento de An\'alisis matem\'atico, Universidad de Granada,
	Campus Fuente-nueva, 18071 Granada, Spain} \email{jingwulx@correo.ugr.es}

\thanks{{\bf Acknowledgements.} D. R. has been supported by the FEDER-MINECO Grant PID2021-122122NB-I00 and by J. Andalucia (FQM-116).
	P. S has been supported by the FEDER-MINECO Grants PID2020-117868GB-I00 and PID2023-150727NB-I00 and by J. Andalucia Grant P18-FR-4049. J. W. has been supported
	by the China Scholarship Council (CSC201906290013) and by J. Andalucia (FQM-116).
	D. R. and P. S. also acknowledge financial support from the Spanish Ministry of Science and Innovation (MICINN), through the \emph{IMAG-Maria de Maeztu} Excellence Grant CEX2020-001105-M/AEI/10.13039/501100011033.}
\maketitle

\noindent

\noindent
\begin{abs}
In this paper, we establish a Modica type estimate on bounded solutions to the overdetermined elliptic problem
\begin{equation*}
  \begin{cases}
  \Delta u+f(u) =0& \mbox{in $\Omega$, }\\
u>0 &\mbox{in $\Omega$, }\\
  u=0 &\mbox{on $\partial\Omega$, }\\
  \partial_{\nu} u=-\kappa &\mbox{on $\partial\Omega$, }
  \end{cases}
\end{equation*}
where $\Omega\subset\mathbb{R}^{n},n\geq 2$. As we will see, the presence of the boundary changes the usual form of the Modica estimate for entire solutions. We will also discuss the equality case. From such estimates we will deduce information about the curvature of $\partial \Omega$ under a certain condition on $\kappa$ and $f$. The proof uses the maximum principle together with scaling arguments and a careful passage to the limit in the arguments by contradiction.
\end{abs}

\medskip

\noindent
\begin{key words}: Overdetermined boundary conditions; Modica type estimate; maximum principle.
\end{key words}

\medskip

 \indent

\section{Introduction and statement of the results}
\label{Section 1}

Overdetermined elliptic problems in the form
\begin{gather}\label{eq11}
\begin{cases}
  \Delta u+f(u)=0 &\mbox{in $\Omega$, } \\
  u>0 &\mbox{in $\Omega$, } \\
  u=0 &\mbox{on $\partial \Omega$, }\\
  \partial_{\nu} u=-\kappa&\mbox{on $\partial \Omega$, }
  \end{cases}
\end{gather}
have received a lot of attention in the last years. Here $\Omega$ is a regular domain (bounded or not) of $\mathbb{R}^n$, $n\geq2$, $f$ is a Lipschitz continuous function, $\nu$ stands for the outward unit vector about $\partial \Omega$ and $\kappa$ is a constant. Since we ask the function $u$ to be positive, it is clear that $-\kappa$ must be non positive. These problems are called overdetermined due to the presence of two boundary conditions and have various applications because they appear as the equilibrium condition in many physical problems, see \cite{S02} and references therein for more details.

\medskip

The study of these problems started in the 70's, when J. Serrin \cite{S71} proved that if $\Omega$ is bounded, $f$ is $C^1$ and problem (\ref{eq11}) admits a solution, then $\Omega$ must be a ball. The method used by Serrin is universally known as the moving plane method and holds also when $f$ is only Lipschitz continuous \cite{PS07}. The case when $\Omega$ is supposed to be unbounded has been considered in 1997 by H. Berestycki, L. Caffarelli and L. Nirenberg \cite{BCN97}, who stated the following:

\medskip

\textbf{BCN Conjecture} (\cite{BCN97}): If $\Omega$ is a smooth domain, $\mathbb{R}^n\setminus \overline{\Omega}$ is connected, and there exists a bounded solution of \eqref{eq11}, then $\Omega$ is either a ball, a half-space, a generalized cylinder $B^k\times \mathbb{R}^{n-k}$ where $B^k$ is a ball in $\mathbb{R}^k$, or the complement of one of them.

\medskip
It turns out that the BCN conjecture is false, and in the last decade many counterexamples have been built using perturbation theory. For example, for domains with the topology of a cylinder, we have the existence of periodic and rotationally symmetric perturbations of $B_1 \times \mathbb{R}$ that support a solution to the problem \eqref{eq11} when $f(u)=\lambda u, \lambda>0$ or $f$ is such that there exists a nondegenerate solution of \eqref{eq11} in $B_1$, see \cite{S10, SS12, RSW22}. For domains with the topology of a half-space, in $\mathbb{R}^9$ the Bombieri-De Giorgi-Giusti epigraph can be perturbed to obtain nontrivial domains $\Omega$ that support solutions to Problem \eqref{eq11} when $f(u)$ is of Allen-Cahn type (i.e. $f(u)=u-u^3$ or other nonlinearities with similar behavior), see \cite{dPPW15}. For the topology of the complement of a ball, for any $n\geq 2$ there exists a ball $B_R$ of radius $R$ large enough such that $\mathbb{R}^n \backslash B_R$ can be perturbed in domains that support solutions to Problem \eqref{eq11} when $f(u)$ is the Schr\"{o}dinger nonlinearity (i.e. $f(u)=u^p-u$, $p<\frac{n+2}{n-2}$). These new solutions in perturbations of the exterior of a ball produce automatically new solutions in perturbations of the complement of a straight cylinder just by adding one or more empty variables.

\medskip

These results show that the geometry of domains where one can solve \eqref{eq11} is in general richer than expected. And this depends on the geometry of $\Omega$, on the equation (i.e. the function $f$) and in a certain way also on the value of the constant $\kappa$. It is therefore natural to investigate rigidity results establishing that the BCN conjecture is true under suitable assumptions on the geometry of $\Omega$, the constant $\kappa$ or the behavior of the function $f$. Serrin's theorem is a rigidity result in this sense, that holds for any $f$ and any constant $\kappa$, just assuming the boundedness of the domain. In 1997 W. Reichel  \cite{R97} proved that if $f$ is of Allen-Cahn type or identically 0 then the existence of a bounded solution to \eqref{eq11} in an exterior domain $\Omega$ implies that $\Omega$ is the complement of a ball (see \cite{RRS20}). In 2013 A. Ros and the second author \cite{RS13} gave an affirmative answer to the BCN conjecture in dimension 2 under one of the following assumptions: the domain $\Omega$ is narrow (i.e. there exists $R>0$ such that $\Omega$ does not contain any ball of radius $R$), or $f(u) \geq \lambda u$ for some $\lambda >0$. In 2017 A. Ros, the first and second authors \cite{RRS17} proved the BCN conjecture for $n=2$, under the assumptions $\kappa >0$ and $\Omega$ diffeomorphic to a half-plane. In 2018, K. Wang and J. Wei \cite{WW19} proved the BCN conjecture under the assumptions that $\Omega$ is a globally Lipschitz epigraph and $f$ is of Allen-Cahn type, generalizing previous results in \cite{BCN97, FV10}.

\medskip


One of the results of this paper is exactly a rigidity result for overdetermined elliptic problems, that roughly speaking is the following: if $f$ and $\kappa$ satisfies a certain condition then either the mean curvature of $\partial \Omega$ is negative, or $\Omega$ is a half-space. More precisely, if $H(p)$ is the mean curvature of $\partial \Omega$ at $p$, we will prove the following rigidity result:

\begin{theorem}\label{coro}
	Let $u\in C^{3}(\Omega)$ be a bounded solution to the problem (\ref{eq11}), with $f \in C^1$. If there exists a non positive primitive $F$ of $f$ such that
	\begin{equation} \label{cond1} \kappa^{2}+2F(0) \geq 0\,,\end{equation}
	then either $H(p) < 0$ for any $p \in \partial\Omega$, or $\Omega$ is a half-space and $u$ is $1$-dimensional, i.e., there exists $x_0 \in \R^n$, a function $g: [0, +\infty) \rightarrow\mathbb{R}$ and $\vec{a}\in \mathbb{R}^{n}$, such that
	\[  \Omega= \{x \in \R^n:  \ \vec{a} \cdot (x-x_0) >0 \}, \ \mbox{ and }u(x)=g(\vec{a}\cdot (x-x_0)), \ \ x \in \Omega.\]
\end{theorem}

\medskip

 Notice that we are considering bounded solutions of \eqref{eq11}, so $f$ is defined in a closed interval and it is always possible to choose a non positive primitive of $f$, by changing $f$ apropriately outside the image of $u$. In particular, if \eqref{cond1} is satisfied, one cannot solve Problem \eqref{eq11} in a ball, nor in a cylinder, nor in the perturbation of a cylinder. A particular case of Theorem \ref{coro} is Theorem 2.13 in \cite{RS13} for double periodic domains in the plane.

\medskip

We will obtain Theorem \ref{coro} as a corollary of more general results, that represent Modica type estimates for overdetermined problems. Given a primitive $F$ of $f$, we define the $P$-function:
\begin{equation}\label{eq22}
	P(x)=|\nabla u(x)|^{2}+2F(u(x)).
\end{equation}  In \cite{M85} L. Modica proved that if $F$ is a nonpositive function and  $u : \mathbb{R}^{n} \rightarrow \mathbb{R}$ is a bounded $C^{3}$ solution of
$$\Delta u + f(u) = 0,$$
then $P \leq0.$ The proof of this result is based on the fact that $P$ satisfies $L(P)\geq 0$ for a certain elliptic operator $L$, see Lemma \ref{LedefP}, and hence the conclusion follows (not immediately) from the maximum principle. This result has been extended to more general operators in \cite{CGS94}, where Caffarelli, Garofalo and Segala provide the following assertion:
$$P(x_{0}) = 0 \,\ \mbox{for some $x_{0} \in \mathbb{R}^{n}$} \Leftrightarrow P(x) = 0 \,\ \mbox{for all $x \in \mathbb{R}^{n}$} \Leftrightarrow u \,\ \mbox{is 1-dimensional}.$$
Modica type estimates have been studied and extended to many other situations, and it is not possible to give here a complete account of all of them. Let us just mention that such estimates have been extended to domains $\Omega$ which are epigraphs with boundary of nonnegative mean curvature in \cite{FV10}, and to compact manifolds with nonnegative Ricci tensor, see \cite{FV11}.

\medskip

Modica type estimates for overdetermined problems have been studied in \cite{W17} under some conditions on the nonlinearity $f$ and the constant $\kappa$. It is our aim in this paper to study the validity of such estimates in full generality. Our main result is the following:

\begin{theorem} \label{Th11}
Let $\Omega\subset\mathbb{R}^{n},\ n\geq 2,$ be a regular domain, $F\in C^{2}(\mathbb{R})$ be a non-positive function, $F'=f,\, u\in C^{3}(\Omega)$ be a bounded solution to the problem (\ref{eq11}) and $P$ be given by (\ref{eq22}). Then
\[P(x)\leq\max\{0,\kappa^{2}+2F(0)\} \ \mbox{ for all } x \in \Omega ,\]
where $P$ is defined in \eqref{eq22}. Moreover, if there exists a point $x_{0}\in\Omega$ such that
 \[P(x_{0})=\max\{0,\kappa^{2}+2F(0)\},\]
then $P$ is constant, $u$ is $1$-dimensional and $\Omega$ is a half-space.
\end{theorem}

The second part of Theorem \ref{Th11} is related to \cite{FV10} (see also \cite{FO21} for generalization to Riemannian manifolds). In the case where $P$ is bounded above by $\kappa^{2}+2F(0)$ we can give information on the mean curvature of $\partial \Omega$, as in the following statement:
\begin{theorem} \label{Th12}
 Let $\Omega\subset\mathbb{R}^{n},n\geq 2,$ be a regular domain that supports a bounded solution $u\in C^{3}(\Omega)$ to the problem (\ref{eq11}) with $\kappa \neq 0$. Assume that
\begin{equation} \label{cond} P(x) \leq \kappa^{2}+2F(0) \ \mbox{ for all } x \in \Omega.\end{equation}
Then, $H(p) \leq 0$ for any $p \in \partial\Omega$. Moreover, if there exists $p \in \partial \Omega$ such that $H(p)=0$, then $P$ is constant, $u$ is $1$-dimensional and $\Omega$ is either a half-space or a slab, that is, the domain between two parallel hyperplanes.
\end{theorem}

Let us point out that Theorem \ref{Th12} does not require that $F$ is nonpositive, since the condition \eqref{cond} is invariant under addition of constants to $F$. On the contrary, it requires that the normal derivative on $\partial \Omega$ does not vanish in order to exclude the presence of critical points on $\partial \Omega$. It is well known that there exist solutions of Problem \eqref{eq11} in balls, cylinders, or generalized onduloids (see for for instance \cite{S10}). Theorem \ref{Th12} implies that in all such cases,
  $$\sup\limits_{x\in \Omega} P(x)>\kappa^{2}+2F(0)\,.$$


\medskip

The proof of the results is organized as follows. As a first ingredient for Theorem \ref{Th11}, we prove a uniform gradient bound on $u$. This is based on a rescaling argument and a Harnack inequality, together with regularity estimates. Once this is obtained, we make use of the fact that the function $P$ is a subsolution for a certain elliptic operator outside the critical points of $u$. For the study of the supremum of $P$ we need to study several cases, depending on the behavior of the maximizing sequences. The proof of Theorem \ref{Th11} concludes again by a scaling argument and passing to a limit, in the spirit of \cite{JK15}. Theorems \ref{Th12} and \ref{coro} will be proved afterwards.

\section{Preliminaries}
\label{Section 2}

Let $u$ be a $C^3$ solution of \eqref{eq11}, $f \in C^1$, and $P$ as in \eqref{eq22} where $F$ is a primitive of $f$ (no assumption on its sign at the moment). We begin this section by showing that the function $P$ is a subsolution of an elliptic PDE. This is a very well known result (see for instance \cite{S81}), which we present here for the sake of completeness.

\begin{lemma} \label{LedefP}The function $P$ satisfies:
	\begin{equation}\label{eqP}
		L(P):= \Delta P+2f(u)\frac{\nabla u}{|\nabla u|^{2}}\cdot\nabla P \geq 0
	\end{equation}	
	for any $x \in \Omega$ such that $\nabla u(x) \neq 0$.
\end{lemma}

\begin{proof}
	We have that $P\in C^{2}(\Omega)$ and following \cite{S81} we find that
	\[D_{i}P=2\sum\limits_{j=1}^{n}D_{j}uD_{ij}u+2f(u)\cdot D_{i}u\]
	for every $i=1,\cdots,n$. Therefore,
	\[D_{i}P-2f(u)\cdot D_{i}u=2\sum\limits_{j=1}^{n}D_{j}uD_{ij}u\leq2|\nabla u|\left[\sum\limits_{j=1}^{n}(D_{ij}u)^{2}\right]^{1/2}.\]
	Using $-\Delta u=f(u),$ the Laplacian of $P$ can be given by
	\begin{align*}
		\Delta P&=2\sum\limits_{i,j=1}^{n}(D_{ij}u)^{2}+2\sum\limits_{j=1}^{n}D_{j}uD_{j}(\Delta u)+2f'(u)|\nabla u|^{2}-2f^{2}(u) \\
		&=2\sum\limits_{i,j=1}^{n}(D_{ij}u)^{2}-2f^{2}(u),
	\end{align*}
	so that
	\[|\nabla u|^{2}\Delta P\geq \frac{1}{2}|\nabla P|^{2}-2f(u)\nabla u\cdot\nabla P,\]
	and then
	\begin{equation*}
		L(P):= \Delta P+2f(u)\frac{\nabla u}{|\nabla u|^{2}}\cdot\nabla P\geq \frac{|\nabla P|^{2}}{2|\nabla u|^{2}} \geq 0.
	\end{equation*}
This concludes the proof of the lemma.
\end{proof}

We finish this section with a uniform estimate of the gradient of $u$. This will be essential in the proof of Theorem \ref{Th11}.
\begin{lemma} \label{Le21} If $u$ is a bounded solution of the problem (\ref{eq11}), then there exists a constant $M>0$ such that $|\nabla u|\leq M$ in $\Omega$. Moreover, $M$ depends only on $n$, $\|u\|_{L^{\infty}}$,  $\| f(u) \|_{L^{\infty}}$ and $\kappa$.
\end{lemma}
\begin{proof}
Let us observe that the above gradient estimate is an immediate consequence of interior regularity estimates for elliptic equations in
\[
\{x\in\Omega:\mbox{dist}(x,\partial\Omega)>\delta\}\,,
\]
for any $\delta >0$ fixed. We now prove that this estimates holds also up to the boundary.

\medskip

Given any $x_{0} \in \Omega$, denote by $h:=\mbox{dist}(x_{0},\partial\Omega)$ and assume that this distance is attained at $y_{0}\in\partial\Omega$. By the previous comment, we can focus on the case $h \leq 1$. We define:
\[\tilde{u}(x):=\frac{1}{h}u(x_{0}+hx),\]
then in $B_{1}(0),$
\[-\Delta\tilde{u}=hf(h\tilde{u}),\,\ \tilde{u}>0.\]
Let $z_{0}:=\frac{y_{0}-x_{0}}{h},$
\[|\nabla\tilde{u}(z_{0})|=|\nabla u(y_{0})|=\kappa,\]
since the ball $B_{h}(x_{0})$ is tangent to $\partial\Omega$ at $y_{0}.$
We now decompose $\tilde{u}=v+w,$ where the functions $v$, $w$ solve the following two problems, respectively:
\begin{equation*}
  \begin{cases}
  \Delta v=0 &\mbox{in $ B_{1}(0)$, }\\
   v=\tilde{u} &\mbox{on $\partial B_{1}(0)$, }
  \end{cases}
\quad \mbox{and} \quad
  \begin{cases}
  -\Delta w=hf(h\tilde{u}) &\mbox{in $ B_{1}(0)$, }\\
 w=0 &\mbox{on $\partial B_{1}(0)$. }
  \end{cases}
\end{equation*}
For the function $w,$ we can use regularity estimates (see for instance Theorem 8.33 of \cite{gt}) to conclude that
\[\|w\|_{C^{1,\alpha}(\bar{B}_{1})} \leq C h\|f(u)\|_{L^{\infty}}.\]
Here $C$ is a constant depending only on $n$.
We now turn our attention to $v$. By applying an explicit version of the Harnack inequality for harmonic functions (see, for instance, \cite[Chapter 2, exercise 2.6]{gt}), one can get that
\begin{equation} \label{harnack} \frac{1-r}{(1+r)^{n-1}}v(0)\leq v(x)\leq\frac{1+r}{(1-r)^{n-1}}v(0)\end{equation}
where $r=|x|<1.$ Using the first inequality and computing boundary derivatives on $z_0$, we have:
$$ 2^{1-n} v(0) \leq \big | \partial_{\nu} v (z_0) \big | \leq |\nabla \tilde{u}(z_0)| + |\partial_\nu w(z_0)| \leq \kappa +C\, h \| f(u) \|_{L^{\infty}}.$$
If $h$ is sufficiently small, we have that
\[v(0)\leq M\,,\]
where $M= 2^{n-1} \Big (\kappa + C \| f(u) \|_{L^{\infty}} \Big )$.
We now use the second inequality in \eqref{harnack} to conclude that \[v(x)\leq M\frac{1+r}{(1-r)^{n-1}}.\]
Based on these results, we get that
 \[\tilde{u}(x)=v(x)+w(x)\leq M\frac{1+r}{(1-r)^{n-1}}+C h \| f(u) \|_{L^{\infty}}.\]
Therefore,
\[\sup\limits_{B_{1/2}(0)}\tilde{u}(x)\leq M,\]
taking a bigger $M$ if necessary. By standard interior gradient estimates,
\[|\nabla u(x_{0})|=|\nabla\tilde{u}(0)|\leq C M,\]
where, again, this constant $C$ depends only on $n$ and $\|f(u) \|_{L^{\infty}}$.
\end{proof}


\section{Proof of Theorems}
\label{Section 3}
In this section, we shall present the proof of our main result, Theorem \ref{Th11}, from which we will deduce Theorem \ref{Th12} and then Theorem \ref{coro}. Define
\begin{equation}\label{eq31}
  \alpha :=\max\{0,\kappa^{2}+2F(0)\}.
\end{equation}

\begin{proposition} Under the assumptions of Theorem \ref{Th11}, we have that $P(x) \leq \alpha$ for all $x \in \Omega$. \end{proposition}

\begin{proof}
By Lemma \ref{Le21} we know that $P$ is bounded. Reasoning by contradiction, assume that
\[\beta:=\sup\limits_{\Omega} P(x)>\alpha.\]
There exists a sequence $(x_{k})_{k\in\mathbb{N}}\subset\Omega$ such that
\[P(x_{k})\rightarrow\beta\,\ \mbox{as}\,\ k\rightarrow \infty.\]
 If $x_{k}$ is bounded, up to a subsequence we can assume that $x_k \to x_{0}$, with $P(x_0)=\beta$. Since $\beta>0$, we conclude that $x_0$ is not a critical point of $u$. Moreover, since $\beta > \kappa^2 + 2 F(0)$, then $x_0 \in \Omega$. As a consequence, $P$ attains a local maximum at $x_0$, and $\nabla u(x_0) \neq 0$; the maximum principle applied to $P$ implies  that $P$ is constant on a neighborhood of $x_0$. This argument implies that the set:
 $$ \{x \in \Omega: \ P(x)= \beta\}$$
 is a non-empty open subset of $\Omega$. Since it is obviously closed, then $P(x)= \beta$ for all $x \in \Omega$. But $P(x)= \kappa^2 + 2F(0) < \beta$ if $x \in \partial \Omega$, and this is a contradiction with the continuity of $P$.

We now assume that $x_{k}$ is unbounded, and discuss the following two cases.

 \medskip

\textbf{Case 1:} $\limsup_{k \to +\infty} \mbox{dist}(x_{k},\partial\Omega)>\delta.$ \\
Let us consider $u_k$ extended by $0$ outside $\Omega$, so that $u_k: \R^n \to \R$ is a globally Lipschitz function. We consider $u_{k}(x)=u(x+x_{k})$ and $P_{k}(x)=P(x+x_{k}),$ so that
\begin{equation}\label{eq12}
P_{k}(0)=P(x_{k})\to\beta.
\end{equation}
By a Cantor diagonal argument we can take a subsequence, still denoted by $u_k$, so that on compact sets of $\R^n$ the sequence $u_k$ converges uniformly to a certain Lipschitz function $u_\infty \geq 0$. Moreover $\nabla u_k \stackrel{*}{\rightharpoonup} \nabla u_{\infty}$, where $\stackrel{*}{\rightharpoonup}$ denotes convergence with the weak star topology in $L^{\infty}$.

Denoting $ P_{\infty}= |\nabla u_{\infty}|^2 + 2 F(u_\infty) \in L^{\infty}(\R^n)$, we have that $P_\infty \leq \beta$. Recall that by Lemma \ref{Le21}, $f(u_k)$ is a globally Lipschitz function in $\R^n$. By interior regularity estimates we have that $u_k$ is bounded in $C^{2, \alpha}$ norm on any compact subset of $B_{\delta}(0)$. Then $u_k$ converges to $u_{\infty}$ in $C^{2, \alpha}$ sense in compact subsets of $B_{\delta}(0)$. In particular,
\begin{equation} \label{jo} 0< \beta \leftarrow P_k(0) \to P_{\infty}(0). \end{equation}

Since $\beta>0$, then $0$ is not a critical point of $u_{\infty}$. This, in particular, excludes the possibility $u_{\infty}(0)=0$. Define the open set $\Omega_{\infty}= \{x \in \R^n: \ u_\infty (x) >0\}$. Clearly $\Omega_{\infty}$ is not empty since $0 \in \Omega_{\infty}$. We first claim that inside $\Omega_\infty$ the convergence is $C^{2, \alpha}_{loc}$, and $-\Delta u_{\infty} = f(u_\infty)$. Indeed, given $p \in \Omega_{\infty}$, by uniform convergence there exists $r>0$ such that $u_k(x)>0$ for any $x \in B_r(p)$. Reasoning as above, interior regularity estimates allow us to conclude the claim.

By applying the maximum principle to $P_{\infty}$, we conclude that $P_{\infty}$ is constant on a neighborhood of $0$. Indeed, we have that the set $\{ x \in \Omega:\ P_{\infty}(x)= \beta \}$ is an open subset of $\Omega_{\infty}$. By continuity of $P_{\infty}$ in $\Omega_{\infty}$, it is also closed. As a consequence,
\begin{equation} \label{jo2} P_{\infty}(x) = \beta \mbox{ for all } x \in \tilde{\Omega}_{\infty}, \end{equation} where $\tilde{\Omega}_{\infty}$ is the connected component of $\Omega_{\infty}$ containing $0$.

Take now $y_k \in \tilde{\Omega}_{\infty}$ such that $u_{\infty}(y_k)  \to \xi$, where
$$ \xi= \sup \{u_{\infty}(x):\ x \in \tilde{\Omega}_{\infty}\}>0.$$ By the Lipschitz regularity of $u_\infty$, there exists $r>0$ such that $B_r(y_k) \subset  \tilde{\Omega}_{\infty}$. Thanks to the Ekeland Variational Principle (see for instance \cite[Chapter 5]{struwe}) there exists $z_k \in \tilde{\Omega}_{\infty}$ such that:
\begin{enumerate}
	\item[i)] $u_{\infty}(z_k) \to \xi$;
	\item[ii)] $|y_k-z_k| \to 0;$
	\item[iii)]  $\nabla u_{\infty}(z_k) \to 0.$
\end{enumerate}
Observe that iii) can be derived only because $B_r(y_k) \subset  \tilde{\Omega}_{\infty}$. In particular, $P_{\infty}(z_k) \to 2 F(\xi) \leq 0$, which is a contradiction with \eqref{jo2}.
\medskip

\textbf{Case 2:}  $\lim_{k \to +\infty} \mbox{dist}(x_{k},\partial\Omega)=0.$ \\
As in case 1, we consider $u_k$ extended by $0$ outside $\Omega$, so that $u_k: \R^n \to \R$ is a globally Lipschitz function. Observe that $u(x_{k})\rightarrow 0$ since $|\nabla u|$ is bounded by Lemma \ref{Le21}. Then,
\begin{equation} \label{jo3} |\nabla u(x_{k})|^{2} = P(x_k) - 2 F(u(x_k)) \to \beta -2F(0) >\kappa^{2}.\end{equation}
Denote $h_{k}:=\mbox{dist}(x_{k},\partial\Omega)\rightarrow0$ and assume that this distance is attained at $y_{k}\in\partial\Omega.$ Let
\[u_{k}(x):=\frac{1}{h_{k}}u(y_{k}+h_{k}x).\]
Observe that $u_k(0)=0$ and $\nabla u_k$ is uniformly bounded by Lemma \ref{Le21}. Therefore, we can take a subsequence such that $u_k$ converges uniformly in compact sets to a limit function $u_\infty \geq 0$, which is Lipschitz continuous (but possibly unbounded). Moreover, $\nabla u_{k}\stackrel{*}{\rightharpoonup}\nabla u_{\infty}$ in $L^{\infty}$.

Let us denote $\Omega_{\infty}= \{x \in \R^n: \ u_{\infty}(x) >0\}$. First, we show that in compact sets of $\Omega_{\infty}$, $u_k$ converges to $u_{\infty}$ in $C^{2, \alpha} $ sense and
$$ \Delta u_{\infty} =0, \ \ x \in \Omega_{\infty}.$$
Indeed, take $p \in \Omega_{\infty}$. By uniform convergence, there is $r>0$ such that $u_k >0$ in $B_{r}(p)$. Recall that $f(u_k)$ is uniformly Lipschitz continuous in $\R$: by interior regularity estimates, we conclude that $u_k \to u_{\infty}$ in $C^{2,\alpha}$ sense in compact sets of $B_r(p)$. Passing to the limit, $u_\infty$ is harmonic in $\Omega_\infty$.

Now, let $z_{k}:=\frac{x_{k}-y_{k}}{h_{k}},$ then we can assume that $z_{k}$ converges to $z_{\infty}$ since $|z_{k}|=1.$ The same argument as above works in $B_{r}(z_\infty)$ for any $r \in (0,1)$, and then $u_k \to u_{\infty}$ in $C^{2,\alpha}$ sense in compact sets of $B_1(z_\infty)$. Hence $u_\infty \geq 0$ is harmonic in $B_1(z_\infty)$. It is clear that, by \eqref{jo3},
\begin{equation}\label{eqg1}
 \kappa < |\nabla u(x_{k})|= |\nabla u_{k}(z_{k})| \to |\nabla u_{\infty}(z_\infty)| := a,
\end{equation}
as $k$ is large.

By the maximum principle we conclude that $u_{\infty}>0$ in $B_1(z_\infty)$, that is, $B_1(z_\infty) \subset  \Omega_{\infty}$.

Observe now that, for all $ x \in B_1(z_\infty)$,
\begin{align*}
      \beta &=|\nabla u_{k}(z_{k})|^{2}+2F(u_{k}(z_{k}))+o_{k}(1)\\
      &=|\nabla u(x_{k})|^{2}+2F(u(x_{k}))+o_{k}(1)\\
      &\geq |\nabla u(y_{k}+xh_{k})|^{2}+2F(u(y_{k}+xh_{k}))\\
      &=|\nabla u_{k}(x)|^{2}+2F(u(y_{k}+xh_{k}))
    \end{align*}
where $F(u(x_{k}))\rightarrow F(0)$ and $F(u(y_{k}+xh_{k}))\rightarrow F(0)$ since $h_{k}\rightarrow 0,u(y_{k})=0$.
As a consequence, $|\nabla u_{\infty}(z_\infty)| \geq |\nabla u_{\infty} (x)|$ for all $ x \in B_1(z_\infty)$. Since $u_{\infty}$ is harmonic, $|\nabla u_{\infty}|$ is sub-harmonic; indeed,
$$  \Delta |\nabla u_{\infty} |^2 = 2 |D^2 u_{\infty}|^2.$$
By the maximum principle $|\nabla u_{\infty} |$ is constant and $D^2 u_{\infty}=0$. Recall moreover that $u_\infty$ is Lipschitz continuous in $\R^n$ and $u_{\infty}(0)=0$.
Therefore, up to a rotation and a translation, $z_{\infty}=0$ and
\[u_{\infty}(x)=ax^{+}_{n},\]
where $a$ is given in \eqref{eqg1}. In particular, $\Omega_{\infty}$ is the upper half-space.

For any $\epsilon\in(0,1)$, by the uniform convergence of $u_{k}$ to $u_{\infty}$,
 \[|u_{k}-ax_{n}^+|<\frac{a\epsilon}{2}\]
in $B_{1}(0)$ for all large enough $k$. Following \cite[Lemma 4.4]{JK15}, let us consider the domain $D_{t},$ the perturbation of $B_{1}(0)\cap\{x_{n}>\epsilon\},$ given by
\[D_{t}=\{x\in B_{1}(0):x_{n}>\epsilon-t\eta (x')\},\]
where $x'=(x_{1},\cdot\cdot\cdot,x_{n-1}),0\leq\eta (x')\leq1$ is a smooth bump function supported in $|x'|\leq\frac{1}{2}$ with $\eta (x')=1$ for $|x'|\leq\frac{1}{4}.$ It is clear that $u_k >0$ in  $D_{0}$ for sufficiently large $k$.

\begin{figure}[htbp]
\centering
\includegraphics[width=5.5cm]{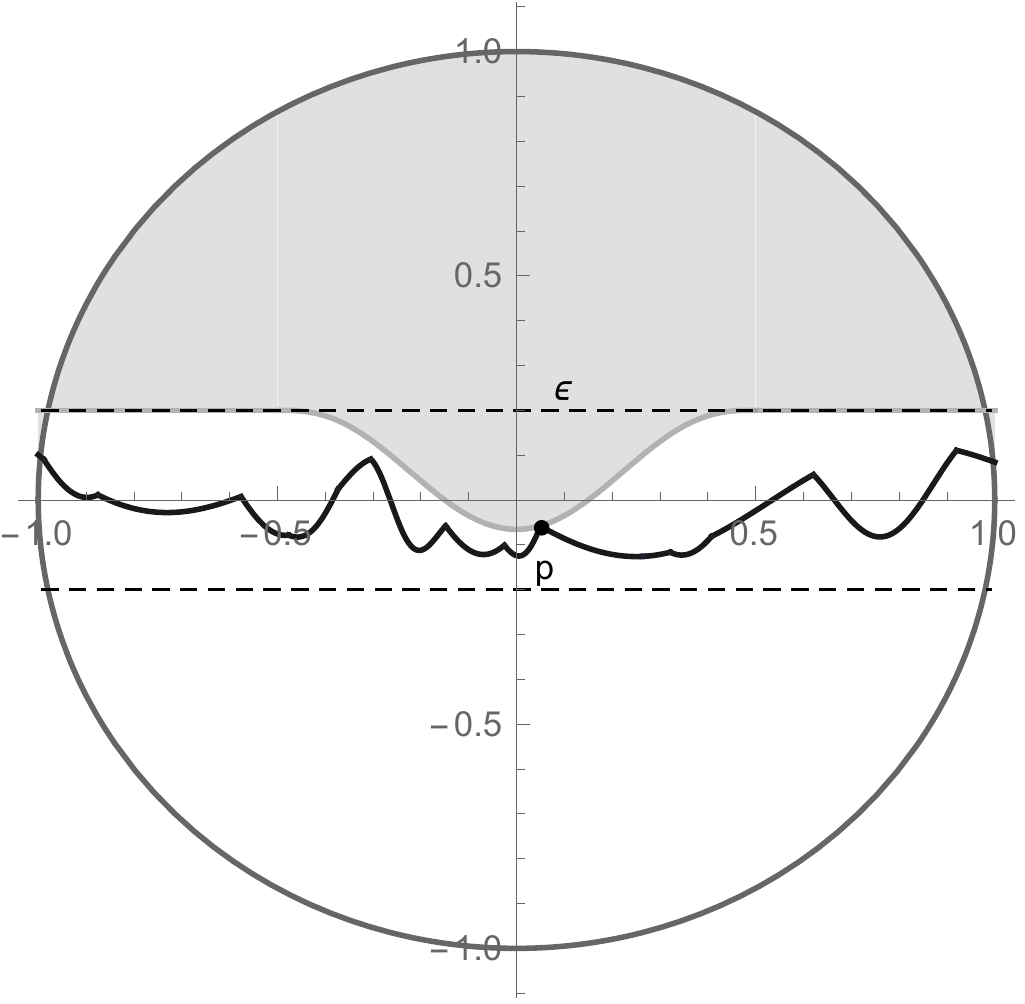}
\caption{The domains $D_{t}$ and $J$, and the point $p$.}
\end{figure}

Now, we denote a function $w$ which solves the following problem
\begin{equation*}
 \begin{cases}
 -\Delta w(x)= -\epsilon &\mbox{in $D_{t_{0}}$},\\
w(x)=ax_{n}-a\epsilon &\mbox{on $\partial B_{1}(0)\cap\{x_{n}>\epsilon\}$},\\
 w(x)=0&\mbox{on $ B_{1}(0)\cap\{x_{n}=\epsilon-t_{0}\eta (x')\}$},
\end{cases}
\end{equation*}
for some $0<t_{0}<2\epsilon,$ where $D_{t_{0}}\subseteq B_{2}(0) \cap \{ u_k>0\}$ will touch
$$J(u_{k}):=B_{2}(0) \cap \partial \{ u_k>0\} $$
at some point $p\in J(u_{k})\cap\{|x'|<\frac{1}{2}\}.$ Thus, $w(x)\leq u_{k}(x)$ in $D_{t_{0}}$ by the maximum principle,
\[|w_{\nu}(p)|\leq |(u_{k})_{\nu}(p)|=\kappa,\]
where $\nu$ is the outer normal to $\partial D_{t_{0}}$ at $p.$

Observe now that the boundary of $D_{t_0}$ is a small deformation in $C^k$ sense (of order $\e$) of the boundary of the half ball. By flattening the boundary we can pass to an elliptic operator posed on the half ball, where the coefficients converge to those of the Laplacian in $C^k$ sense. By using $C^1$ convergence of the soutions up to the boundary, we obtain that
\[|w_\nu(p)|=a+o_\epsilon(1).\]
We therefore conclude $a\leq \kappa$ since $\epsilon$ is arbitrary. This is a contradiction with \eqref{eqg1}.
\end{proof}

In order to conclude the proof of Theorem \ref{Th11} we only need to show the following result.

\begin{proposition} \label{esta} If, under the assumptions of Theorem \ref{Th11}, $P(x)=\alpha$ at a point $x\in\Omega,$ then $P$ is constant, $u$ is $1$-dimensional and $\Omega$ is a half-space. \end{proposition}

\begin{proof}

The proof of this proposition follows the ideas of \cite{CGS94}, and is divided in two steps.

\medskip

\textbf{Step 1:}  If $P(x_0)= \alpha$, then $\nabla u(x_0) \neq 0$.

Reasoning by contradiction, if $\nabla u(x_{0})=0,$ then $0 \leq \alpha=2F(u(x_{0})) \leq 0$, so that $\alpha=0= F(u_0)$, where $u_0 = u(x_0)$. We consider the set
\[U=\{x\in \Omega:u(x)=u_{0}\},\]
where $u_{0}=u(x_{0}).$ Clearly, $U$ is closed and $U\neq\emptyset$. Then we take $x_{1}\in U$ and consider the function $\varphi(t)=u(x_{1}+tw)-u_{0},$ where $w\in \mathbb{S}^{n-1}$ is arbitrarily fixed. We have
\[\varphi'(t)=\nabla u(x_{1}+tw)w.\]
Then
\begin{align*}
      |\varphi'(t)|^{2}&\leq|\nabla u(x_{1}+tw)|^{2}\leq\alpha-2F(u(x_{1}+tw))\\
      &=2F(u_{0})-2F(u(x_{1}+tw)).
    \end{align*}
Since $F(u_{0})=0$, then $F'(u_{0})=0$ and $F''(u_{0})\leq0$ by the $C^2$ regularity of $F$. Hence $F(u)=O((u-u_{0})^{2})$ as $|u-u_{0}|\rightarrow0.$ Therefore we can get that
   $$ |\varphi'(t)|\leq C|\varphi(t)|,$$
 as $t$ small enough. It follows that $\varphi\equiv0$ in $[-\epsilon,\epsilon],$ for some $\epsilon>0,$ since $\varphi(0)=0.$ It shows that the set $U$ given above is open. Then $U=\Omega$, which implies that $u$ is a constant, a contradiction.

 \medskip

 \textbf{Step 2:}  Conclusion.

Then, we have that $\nabla u(x)\neq0$ for any $x\in\Omega$ with $P(x) = \alpha$. By the maximum principle applied to $P$, we conclude that the set
$$ \{ x \in \Omega: \ P(x)= \alpha\}$$ is a non-empty open set, which is also closed by continuity. Then $P(x)= \alpha$ for all $x \in \Omega$. In particular, $u$ has no critical points in $\Omega$.

We now set $v=G(u),$ where $G\in C^{2}(\mathbb{R})$ is suitably determined. By the straightforward computation, one has
\begin{equation} \label{otra}
   \Delta v=G''(u)|\nabla u(x)|^{2}+G'(u)\Delta u =G''(u)(\alpha-2F(u))-G'(u)F'(u).
\end{equation}
   Then we have that
   \begin{equation}\label{eq41}
   \Delta v=0,
   \end{equation}
   if we choose
   \[G(u)=\int_{u_{0}}^{u}(\alpha-2F(s))^{-\frac{1}{2}}ds,\]
   for some fixed $u_{0}\in u(\Omega).$ Observe that $\alpha > 2 F(u_0)$ for any $u_0$ in the range of $u$, so that the integral defining $G$ is not singular. With this choice, one has that
   \begin{equation}\label{eq42}
    |\nabla v|^{2}=G'(u)^{2}|\nabla u|^{2}=1.
\end{equation}
We can infer that $v(x)=\vec{a}\cdot x+b$ for some $\vec{a}\in \mathbb{R}^{n}$ with $|\vec{a}|=1$ and $b\in \mathbb{R}$.
Then, we can obtain that $u(x)=G^{-1}(v(x))=g(\vec{a}\cdot x+b)$ with $g=G^{-1},$ since $G$ is invertible. Then $u$ is $1$-dimensional. A priori, $\Omega$ could be the inner space between two parallel hyperplanes, but this is impossible since $u$ has no critical points by Step 1. This concludes the proof. \end{proof}

\medskip

We now prove Theorems \ref{Th12} and \ref{coro}.

\begin{proof}[Proof of Theorem \ref{Th12}]

Since $\kappa\neq 0$, we have that $u$ has no critical points close to $\partial \Omega$. By the assumption that $P$ attains its maximum at $\partial\Omega,$ then
    \[P_{\nu}\geq0,\quad \mbox{on $\partial\Omega.$}\]
Let $H$ be the mean curvature of $\partial\Omega$ at a given point. On the other hand,
    \begin{align*}
      \frac{\partial P}{\partial\nu} &=\frac{\partial}{\partial\nu}\left(|\nabla u|^{2}+2F(u)\right) \\
      & =2\frac{\partial u}{\partial\nu}\frac{\partial^{2} u}{\partial\nu^{2}}+2f(u)\frac{\partial u}{\partial\nu}\\
      & =2\frac{\partial u}{\partial\nu}\left(\Delta u-(n-1)H\frac{\partial u}{\partial\nu}+f(u)\right)\\
       & =-2(n-1)\Big|\frac{\partial u}{\partial\nu}\Big|^{2}H,
    \end{align*}
    by using the fact that $\frac{\partial^{2} u}{\partial\nu^{2}}=\Delta u-(n-1)H\frac{\partial u}{\partial\nu}$ on $\partial\Omega,$ see \cite[Section 5.4]{S81}. Therefore,
  \[H\leq0.\]
If $H(p)=0$ for some $p \in \partial \Omega$, one has that $\frac{\partial P}{\partial\nu}(p)=0$. By Hopf's lemma, we can get that
$P$ is constant, at least in a neighborhood of $p$. We can now argue as in Step 2 of the proof of Proposition \ref{esta} (see \eqref{otra}, \eqref{eq41}, \eqref{eq42}) to conclude that $u$ is $1$-dimensional in such neighborhood. By unique continuation, $u$ is $1$-dimensional. As a consequence, $\Omega$ is either a half-space or a slab.
\end{proof}

\begin{proof}[Proof of Theorem \ref{coro}] By Theorem \ref{Th11}, we have that $P(x) \leq \kappa^{2}+2F(0)$ for all $x \in \Omega$.

Let us consider first the case $\kappa \neq 0$. In this case we can apply Theorem \ref{Th12} to deduce $H \leq 0$. Moreover, if $H(p)=0$ at some point of $\partial \Omega$, then $P$ is constant and $u$ is $1$-dimensional. Moreover, by Step 1 of Proposition \ref{esta}, $u$ has no critical points, and then $\Omega$ is a half-space.

We now address the case $\kappa=0$. Observe that:
$$ 0 \leq \kappa^2 + 2 F(0) = 2 F(0) \leq 0.$$
As a consequence, $F(0)=0$. Since $F$ is nonpositive, this implies that $f(0)=F'(0)=0$.  But then we obtain a contradiction with the Hopf lemma and this concludes the proof.
\end{proof}


\end{document}